\documentclass[12pt,reqno]{amsart}

\usepackage{amssymb}
\usepackage{amsmath}
\usepackage{latexsym}
\usepackage{amsthm}
\usepackage{epsfig}
\usepackage{color}
\usepackage{comment}
\usepackage{amsfonts,amssymb,mathrsfs,amscd}

\usepackage{enumitem}

\theoremstyle{plain}
\newtheorem{theorem}{Theorem}[section]

\newtheorem{corollary}{Corollary}[section]
\newtheorem{lemma}{Lemma}[section]
\newtheorem{remark}{\bf Remark}[section]
\theoremstyle{definition}

\newcommand{\rot}{\mathop\mathrm{rot}}
\newcommand{\grad}{\mathop\mathrm{grad}}
\renewcommand{\div}{\mathop\mathrm{div}}

\title
[Lieb--Thirring constant  on the sphere and on the torus ]
 {Lieb--Thirring constant  on the sphere and on the torus}

\author[A.Ilyin, A.Laptev, S.Zelik] {Alexei  Ilyin, Ari Laptev, and Sergey Zelik}

\begin{document}
 \begin{abstract}
 We prove on the sphere $\mathbb{S}^2$ and on the  torus
 $\mathbb{T}^2$
  the Lieb--Thirring
inequalities with improved constants  for orthonormal families of  scalar and vector functions.
\end{abstract}


\address
{\noindent\newline (A.A.\,Ilyin and S.V.\,Zelik) Keldysh Institute of Applied Mathematics;
\newline
(A.A.\,Laptev) Imperial College London and St Petersburg University
14-Ya Liniya B.o., 29B St Petersburg 199178, Russia;
\newline
(S.V.\,Zelik) School of Mathematics and Statistics, Lanzhou
University, China  and University of Surrey, Department of
Mathematics, Guildford, GU2 7XH, United Kingdom} \email{\newline
ilyin@keldysh.ru;\newline  a.laptev@imperial.ac.uk;\newline s.zelik@surrey.ac.uk}

\maketitle

\setcounter{equation}{0}
\section{Introduction}\label{sec1}
\label{sec0}
The Lieb--Thirring inequalities~\cite{LT} give estimates for $\gamma$-moments
of the negative eigenvalues of the Schr\"odinger operator $-\Delta-V$ in $L_2(\mathbb{R}^d)$,
where  $V=V(x)\ge0$:
\begin{equation}\label{LT}
\sum_{\lambda_i\le0}|\lambda_i|^\gamma\le\mathrm{L}_{\gamma,d}
 \int_{\mathbb{R}^d} V(x)^{\gamma+ d/2}dx.
\end{equation}
In the case $\gamma=1$ estimate~\eqref{LT} is equivalent to the dual
inequality
\begin{equation}\label{orth}
 \int_{\mathbb{R}^d} \rho(x)^{1+2/d}dx\le\mathrm{k}_d\sum_{j=1}^N\|\nabla\psi_j\|^2,
\end{equation}
where $\rho(x)$ is as in \eqref{rho}, and $\{\psi_j\}_{j=1}^N\in H^1(\mathbb{R}^d)$
is an arbitrary orthonormal system. Furthermore the (sharp) constants $\mathrm{k}_d$ and
 $\mathrm{L}_{1,d}$ satisfy
\begin{equation}\label{kL}
 \mathrm{k}_d=(2/d)(1+d/2)^{1+2/d}\mathrm{L}_{1,d}^{2/d}.
\end{equation}

Sharp constants in  \eqref{LT} were found in~\cite{Lap-Weid} for $\gamma\ge 3/2$,
while for a long time  the best available estimates for  $1\le\gamma<3/2$ were those  found in
\cite{D-L-L}. Very recently an important improvement in the area was made
in~\cite{Frank-Nam}, where the original idea of~\cite{Rumin2}
was developed and extended in a substantial way.

Inequality \eqref{orth} plays an important role in the theory of the
 Navier--Stokes equations
\cite{Lieb, B-V,  T}, where the constant $ \mathrm{k}_2$ enters the estimates
of the fractal dimension of the global attractors
of the Navier--Stokes system in various two-dimensional formulations.
(In the three-dimensional case the corresponding results are  of a conditional character.)

Along with the  problem in a bounded domain
$\Omega\subset\mathbb{R}^2$ with Dirichlet boundary conditions
the Navier--Stokes system is also studied with periodic boundary conditions,
that is, on a two-dimensional torus. In this case for the system to be dissipative
one has to impose the zero mean condition on the components of the velocity vector
over the torus.

Another physically relevant model  is the Navier--Stokes system on the sphere.
In this case the system is dissipative without extra orthogonality conditions.
However, if we want to study the system in the form of the scalar vorticity equation, then the
scalar stream function of a divergence free vector field is defined up
to an additive constant, and without loss of generality we can (and always) assume that
the integral of the stream function over the sphere vanishes.

We can formulate our main result as follows.

\begin{theorem}\label{Th:1}
Let ${M}$ denote either $\mathbb{S}^2$ or $\mathbb{T}^2$,
and let $\dot H^1({M})$ be the Sobolev space of functions with
mean value zero.
Let $\{\psi_j\}_{j=1}^N \in\dot H^1({M})$ be an orthonormal family
in $L_2({M})$.  Then
\begin{equation}\label{rho}
\rho(x):=\sum_{j=1}^N|\psi_j(x)|^2
\end{equation}
satisfies the inequality
\begin{equation}\label{M}
\int_{{M}}\rho(x)^2d{M}\le\mathrm{k}\sum_{j=1}^N\|\nabla\psi_j\|^2,
\end{equation}
where
$$
\mathrm{k}\le\frac{3\pi}{32}=0.2945\dots\,.
$$
\end{theorem}

\begin{corollary} Setting  $N=1$ and $\psi=\varphi/\|\varphi\|$ we obtain
the interpolation inequality which is often called the Ladyzhenskaya
inequality (in the context of the Navier--Stokes equations) or the
Keller--Lieb--Thirring  one-bound-state inequality (in the context of the spectral theory):
$$
\|\varphi\|_{L_4}^4\le \mathrm{k}_{\mathrm{Lad}}\|\varphi\|^2\|\nabla\varphi\|^2,
\qquad \mathrm{k}_{\mathrm{Lad}}\le\mathrm{k}_{\mathrm{LT}}.
$$
\end{corollary}
\begin{remark}
{\rm
The previous estimate of the Lieb--Thirring constant
on $\mathbb{T}^2$ and $\mathbb{S}^2$ obtained in~\cite{Zel-Il-Lap2019} and \cite{I-L-AA} by
means of the discrete version of the method of \cite{Rumin2} was:
$$
\mathrm{k}\le\frac{3}{2\pi}=0.477\,.
$$
}
\end{remark}
\begin{remark}
{\rm
In all cases $M=\mathbb{S}^2, \mathbb{T}^2$, or $\mathbb{R}^2$ the Lieb--Thirring constant
satisfies the (semiclassical) lower bound
$$
0.1591\dots=\frac1{2\pi}\le\mathrm{k}_{\mathrm{LT}}.\ 
$$
In $\mathbb{R}^2$ the sharp value of $\mathrm{k}_{\mathrm{Lad}}$ was found in
\cite{Weinstein} by the numerical solution of the corresponding Euler--Lagrange equation
$$
\mathrm{k}_{\mathrm{Lad}}=\frac1{\pi\cdot 1.8622\dots}=0.1709\dots\,,
$$
while the best to date  closed form estimate for this constant  was obtained in
\cite{Nasibov}
$$
\mathrm{k}_{\mathrm{Lad}}\le\frac{16}{27\pi}=0.188\dots,
$$
see also \cite[Theorem 8.5]{Lieb--Loss} where the  equivalent result is obtained for the inequality
in the additive form.
}
\end{remark}

\setcounter{equation}{0}
\section{Lieb--Thirring inequalities on  $\mathbb{S}^2$ }\label{sec2}

We  begin with the case of a sphere and first consider the scalar
case. We  recall the basic facts concerning the spectrum of the
scalar Laplace operator $\Delta=\div\nabla$ on the sphere
$\mathbb{S}^{2}$:
\begin{equation}\label{harmonics}
-\Delta Y_n^k=n(n+1) Y_n^k,\quad
k=1,\dots,2n+1,\quad n=0,1,2,\dots.
\end{equation}
Here the $Y_n^k$ are the orthonormal real-valued spherical
harmonics and each eigenvalue $\Lambda_n:=n(n+1)$ has multiplicity $2n+1$.

The following identity is essential in what
follows \cite{S-W}: for any $s\in\mathbb{S}^{2}$
\begin{equation}\label{identity}
\sum_{k=1}^{2n+1}Y_n^k(s)^2=\frac{2n+1}{4\pi}.
\end{equation}

\begin{theorem}\label{Th:S2}
Let $\{\psi_j\}_{j=1}^N\in H^1(\mathbb{S}^2)$ be an orthonormal family of scalar functions
with zero average: $\int_{\mathbb{S}^2}\psi_j(s)dS=0$. Then $\rho(s):=\sum_{j=1}^N|\psi_j(s)|^2$
satisfies the inequality
\begin{equation}\label{LTS2}
\int_{\mathbb{S}^2}\rho(s)^2dS\le\frac{3\pi}{32}\sum_{j=1}^N\|\nabla\psi_j\|^2.
\end{equation}
\end{theorem}
\begin{proof}
We use the discrete version of the recent important far-going improvement \cite{Frank-Nam}
of the approach of \cite{Rumin2}.

Let $f$ be a smooth non-negative function on $\mathbb{R}^+$ with
\begin{equation}\label{f}
\int_0^\infty f(t)^2dt=1,
\end{equation}
and therefore for any $a>0$
\begin{equation}\label{fa}
a=\int_0^\infty f(E/a)^2dE.
\end{equation}
Expanding a function $\psi$ with $\int_{\mathbb{S}^2}\psi(s)dS=0$  in spherical harmonics
$$
\psi(s)=\sum_{n=1}^\infty\sum_{k=1}^{2n+1}\psi_n^kY_n^k(s),\qquad
\psi_n^k=\int_{\mathbb{S}^2}\psi(s)Y_n^k(s)dS=(\psi,Y_n^k)
$$
and observing that the summation starts with $n=1$ we see using
\eqref{fa} that
\begin{equation}\label{chain}
\aligned
\|\nabla\psi\|^2=\int_{\mathbb{S}^2}|\nabla\psi(s)|^2dS=
\sum_{n=1}^\infty n(n+1)\sum_{k=1}^{2n+1}|\psi_n^k|^2=\\=
\int_0^\infty \sum_{n=1}^\infty f\biggl(\frac E{ n(n+1)}\biggr)^2\,\sum_{k=1}^{2n+1}|\psi_n^k|^2dE=\\=
\int_0^\infty\int_{\mathbb{S}^2}|\psi^E(s)|^2dSdE=
\int_{\mathbb{S}^2}\int_0^\infty|\psi^E(s)|^2dEdS,
\endaligned
\end{equation}
where
$$
\psi^E(s)=\sum_{n=1}^\infty\sum_{k=1}^{2n+1}f\biggl(\frac E{ n(n+1)}\biggr)
\psi_n^k\,Y_n^k(s).
$$
Returning to the family $\{\psi_j\}_{j=1}^N$ we have for any $\varepsilon>0$
$$
\aligned
\rho(s)&=\sum_{j=1}^N|\psi_j(s)|^2=\\&=\sum_{j=1}^N|\psi^E_j(s)|^2
+2\sum_{j=1}^N\psi^E_j(s)(\psi_j(s)-\psi_j^E(s))+\sum_{j=1}^N|\psi_j(s)-\psi^E_j(s)|^2\le\\&\le
(1+\varepsilon)\sum_{j=1}^N|\psi^E_j(s)|^2+
(1+\varepsilon^{-1})\sum_{j=1}^N|\psi_j(s)-\psi^E_j(s)|^2.
\endaligned
$$
For each term in the second sum we have
$$
\psi(s)-\psi^E(s)=\sum_{n=1}^\infty\sum_{k=1}^{2n+1}\psi_n^k\biggl(1-f\biggl(\frac E{ n(n+1)}\biggr)\biggr)
\,Y_n^k(s)=\bigl(\psi(\cdot),\chi^E(\cdot,s)\bigr),
$$
where
$$
\chi^E(s',s)=\sum_{n=1}^\infty\sum_{k=1}^{2n+1}\biggl(1-f\biggl(\frac E{ n(n+1)}\biggr)\biggr)
\,Y_n^k(s')Y_n^k(s).
$$
Since the $\psi_j$'s are orthonormal, we have by Bessel's inequality
$$
\sum_{j=1}^N|\psi_j(s)-\psi^E_j(s)|^2=\sum_{j=1}^N\bigl(\psi_j(\cdot),\chi^E(\cdot,s)\bigr)^2
\le\|\chi^E(\cdot,s)\|^2,
$$
where in view of \eqref{identity} $\|\chi^E(\cdot,s)\|^2$, in fact, is independent of $s$:
\begin{equation}\label{indeps}
\aligned
\|\chi^E(\cdot,s)\|^2=
\sum_{n=1}^\infty\sum_{k=1}^{2n+1}\biggl(1-f\biggl(\frac E{ n(n+1)}\biggr)\biggr)^2
Y_n^k(s)^2=\\=\frac1{4\pi}\sum_{n=1}^\infty(2n+1)\biggl(1-f\biggl(\frac E{ n(n+1)}\biggr)\biggr)^2.
\endaligned
\end{equation}

We now specify the choice of $f$ by setting (see~\cite{Frank-Nam}, \cite{lthbook})
\begin{equation}\label{choice}
f(t)=\frac1{1+\mu t^2},\qquad\
\mu=\frac{\pi^2}{16}\,.
\end{equation}
The function $f$ so chosen solves the minimization problem
$$
\aligned
\int_{\mathbb{R}^2}\left(1-f(1/|\xi|^2)\right)^2d\xi=&\pi\int_0^\infty(1-f(t))^2t^{-2}dt\to\min\\
\text{under condition}\quad&\int_0^\infty f(t)^2dt=1,
\endaligned
$$
and the above integral over $\mathbb{R}^2$ corresponds to the series
on the right-hand side in \eqref{indeps} (see also~\eqref{tor}).

We first observe that \eqref{f} is satisfied and secondly,
in view of the estimate for the series in the Appendix
\begin{equation}\label{series1}
\aligned
\|\chi^E(\cdot,s)\|^2=\frac1{4\pi}\sum_{n=1}^\infty\frac
{(2n+1)}{\biggl({1+\left(\frac1{\sqrt{\mu}E}n(n+1)\right)^2}\biggr)^2}<\\<
\frac1{4\pi}\sqrt{\mu}E\int_0^\infty\frac{dt}{(1+t^2)^2}=
\frac1{4\pi}\sqrt{\mu}E\frac\pi4=\frac\pi{64}E=:AE\,.
\endaligned
\end{equation}
Hence
\begin{equation}\label{epseps}
\rho(s)\le(1+\varepsilon)\sum_{j=1}^N|\psi^E_j(s)|^2+
(1+\varepsilon^{-1})AE.
\end{equation}
Optimizing with respect to $\varepsilon$ we obtain
$$
\rho(s)\le\left(\sqrt{\sum_{j=1}^N|\psi^E_j(s)|^2}+\sqrt{AE}\right)^2,
$$
which gives that
$$
\sum_{j=1}^N|\psi^E_j(s)|^2\ge\left(\sqrt{\rho(s)}-\sqrt{AE}\right)^2_+.
$$
Summing equalities \eqref{chain} from $j=1$ to $N$ we
obtain
$$
\aligned
&\sum_{j=1}^N\|\nabla\psi_j\|^2=\int_{\mathbb{S}^2}\int_0^\infty
\sum_{j=1}^N|\psi_j^E(s)|^2dEdS\ge\\&
\int_{\mathbb{S}^2}\int_0^\infty\left(\sqrt{\rho(s)}-\sqrt{AE}\right)^2_+dEdS=
\frac1{6A}\int_{\mathbb{S}^2}\rho(s)^2dS=\frac{32}{3\pi}\int_{\mathbb{S}^2}\rho(s)^2dS.
\endaligned
$$
The proof is complete.
\end{proof}
\begin{remark}\label{R:semi}
{\rm
The constant $\mathrm{k}$ in the theorem satisfies the (semiclassical) lower bound
\begin{equation}\label{lb}
k\ge\frac1{2\pi}\,,
\end{equation}
which can easily be proved in our particular case of $\mathbb{S}^2$. In fact, we
take for the orthonormal family the eigenfunctions $Y_n^k$ with $n=1,\dots,N-1$,
and $k=1,\dots,2n+1$, so that
$$
\sum_{n=1}^{N-1}(2n+1)=N^2-1\ \ \text{and}\ \  \sum_{n=1}^{N-1}(2n+1)n(n+1)=\frac12N^2(N^2-1),
$$
then \eqref{M} and the Cauchy inequality give \eqref{lb}, since
$$
(N^2-1)^2=\left(\int_{\mathbb{S}^2}\rho(s)dS\right)^2\le
4\pi\|\rho\|^2\le 
2\pi \mathrm{k}N^2(N^2-1).
$$

}
\end{remark}

\subsection{The vector case}

The vector case is similar, and the key identity~\eqref{identity} is replaced by
vector analogue
(see \cite{I93}): for any $s\in\mathbb{S}^{2}$
\begin{equation}\label{identity-vec}
\sum_{k=1}^{2n+1}|\nabla Y_n^k(s)|^2=n(n+1)\frac{2n+1}{4\pi}.
\end{equation}

In the vector case by the  Laplace operator
acting on (tangent) vector
fields on $\mathbb{S}^2$ we mean  the Laplace--de Rham
operator $-d\delta-\delta d$ identifying $1$-forms and
vectors. Then for a two-dimensional manifold
(not necessarily $\mathbb{S}^2$) we have
\cite{I93}
\begin{equation}\label{vecLap}
\mathbf{\Delta} u=\nabla\div u-\rot\rot u,
\end{equation}
where the operators $\nabla=\grad$ and $\div$ have the
conventional meaning. The operator $\rot$ of a vector $u$ is a
scalar  and for a scalar $\psi$,
$\rot\psi$ is a vector:
\begin{equation}\label{divrot}
\rot u:=\div(u^\perp),\qquad
\rot\psi:=\nabla^\perp\psi,
\end{equation}
where  in the local frame $u^\perp=(u_2,-u_1)$.

 Integrating by parts
we obtain
\begin{equation}\label{byparts}
(-\mathbf{\Delta} u,u)=\|\rot u\|^2+\|\div u\|^2.
\end{equation}

The vector Laplacian has a complete in $L_2(T\mathbb{S}^2)$ orthonormal basis
of vector eigenfunctions:  corresponding
to the eigenvalue
$\Lambda_n=n(n+1)$, where $n=1,2,\dots$, there are two families of $2n+1$
orthonormal vector-valued  eigenfunctions  $w_n^k(s)$ and $v_n^k(s)$
\begin{equation}\label{bases}
\aligned
w_n^k(s)&=(n(n+1))^{-1/2}\,\nabla^\perp Y_n^k(s),\ -\mathbf{\Delta}w_n^k=n(n+1)w_n^k,\
\div w_n^k=0;\\
v_n^k(s)&=(n(n+1))^{-1/2}\,\nabla Y_n^k(s),\ \ -\mathbf{\Delta}v_n^k=n(n+1)v_n^k,\ \rot v_n^k=0,
\endaligned
\end{equation}
where
$k=1,\dots,2n+1$,  and~(\ref{identity-vec}) gives the
following important identities: for any $s\in\mathbb{S}^2$
\begin{equation}\label{id-vec}
\sum_{k=1}^{2n+1}|w_n^k(s)|^2=\frac{2n+1}{4\pi},\qquad
\sum_{k=1}^{2n+1}|v_n^k(s)|^2=\frac{2n+1}{4\pi}.
\end{equation}
We finally observe that  $-\mathbf{\Delta}$ is strictly
positive $-\mathbf{\Delta}\ge \Lambda_1I=2I.$

\begin{theorem}\label{Th:LT-vec}
Let $\{u_j\}_{j=1}^N\in H^1(T\mathbb{S}^2)$
be an  orthonormal family  of vector fields in $L^2(T\mathbb{S}^2)$.  Then
\begin{equation}\label{orthvec}
\int_{\mathbb{S}^2}\rho(s)^2dS\le
 \frac{3\pi}{16}\sum_{j=1}^N(\|\rot u_j\|^2
+\|\div u_j\|^2),
\end{equation}
where $\rho(s)=\sum_{j=1}^N|u_j(s)|^2$.
If, in addition, $\div u_j=0$ $($or $\rot u_j=0$$)$,
then
\begin{equation}\label{orthvecsol}
\int_{\mathbb{S}^2}\rho(s)^2dS\le\frac{3\pi}{32}\cdot
\begin{cases}\displaystyle
\sum_{j=1}^N\|\rot u_j\|^2,
\quad \ \div u_j=0,
\\\displaystyle
\sum_{j=1}^N\|\div u_j\|^2,
\quad \rot u_j=0.
\end{cases}
\end{equation}
\end{theorem}
\begin{proof} We prove the first inequality in~\eqref{orthvecsol},
the proof of the second is  similar. Expanding a vector function
$u$ with $\div u=0$ in the basis $w_n^k$
$$
u(s)=\sum_{n=1}^\infty\sum_{k=1}^{2n+1} u_n^kw_n^k(s),\qquad
u_n^k=(u,w_n^k),
$$
we have instead of \eqref{chain}
\begin{equation}\label{chain1}
\aligned
\|\rot u\|^2&=
\sum_{n=1}^\infty n(n+1)\sum_{k=1}^{2n+1}|u_n^k|^2=\\&=
\int_0^\infty \sum_{n=1}^\infty f\biggl(\frac E{ n(n+1)}\biggr)^2\,\sum_{k=1}^{2n+1}|u_n^k|^2dE=
\int_{\mathbb{S}^2}\int_0^\infty|u^E(s)|^2dEdS,
\endaligned
\end{equation}
where
$$
u^E(s)=\sum_{n=1}^\infty\sum_{k=1}^{2n+1}f\biggl(\frac E{ n(n+1)}\biggr)
u_n^k\,w_n^k(s).
$$
As before
$$
\aligned
\rho(s)\le
(1+\varepsilon)\sum_{j=1}^N|u^E_j(s)|^2+
(1+\varepsilon^{-1})\sum_{j=1}^N|u_j(s)-u^E_j(s)|^2.
\endaligned
$$
We now imbed $\mathbb{S}^2$ into $\mathbb{R}^3$
in the natural way and use the standard basis $\{e_1,e_2,e_3\}$
and the scalar product $\langle \cdot,\cdot\rangle$ in $\mathbb{R}^3$.
Then we see that
$$
\aligned
\langle u(s)-u^E(s),e_1\rangle=\\\sum_{n=1}^\infty\sum_{k=1}^{2n+1}u_n^k\biggl(1-f\biggl(\frac E{ n(n+1)}\biggr)\biggr)
\,\langle w_n^k(s),e_1\rangle=\bigl(u(\cdot),\chi^E_1(\cdot,s)\bigr),
\endaligned
$$
where the vector function
$$
\chi^E_1(s',s)=\sum_{n=1}^\infty\sum_{k=1}^{2n+1}\biggl(1-f\biggl(\frac E{ n(n+1)}\biggr)\biggr)
\,w_n^k(s')\langle w_n^k(s),e_1\rangle.
$$
By orthonormality and Bessel's inequality
$$
\aligned
\sum_{j=1}^N|u_j(s)-u^E_j(s)|^2=\sum_{j=1}^N\sum_{l=1}^3|\langle u_j(s)-u^E_j(s),e_l\rangle|^2=\\
=\sum_{l=1}^3\sum_{j=1}^N\bigl(u_j(\cdot),\chi^E_l(\cdot,s)\bigr)^2
\le\sum_{l=1}^3\|\chi_l^E(\cdot,s)\|^2.
\endaligned
$$
However, in view of~\eqref{id-vec}, the right hand side is again independent of $s$
$$
\aligned
\sum_{l=1}^3\|\chi^E_l(\cdot,s)\|^2&=
\sum_{n=1}^\infty\biggl(1-f\biggl(\frac E{ n(n+1)}\biggr)\biggr)^2\sum_{k=1}^{2n+1}\sum_{l=1}^3
|\langle w_n^k(s),e_l\rangle|^2=\\&=
\sum_{n=1}^\infty\biggl(1-f\biggl(\frac E{ n(n+1)}\biggr)\biggr)^2\sum_{k=1}^{2n+1}
| w_n^k(s)|^2=\\&=\frac1{4\pi}\sum_{n=1}^\infty(2n+1)\biggl(1-f\biggl(\frac E{ n(n+1)}\biggr)\biggr)^2,
\endaligned
$$
and we complete the proof in exactly the same way as we have
done in the proof of Theorem~\ref{Th:S2} after \eqref{indeps}.
Finally, in the proof of inequality~\eqref{orthvec} both families
of vector eigenfunctions \eqref{bases} play equal roles,
and the constant is increased by the factor of two.
\end{proof}

This, however, does not happen for a single vector function.

\begin{corollary}\label{C:vec}
Let $u\in H^1(T\mathbb{S}^2)$. Then
\begin{equation}\label{vecLad}
\|u\|^4_{L_4}\le \vec{\mathrm{k}}_\mathrm{Lad}
\|u\|^2\left(\|\rot u\|^2+\|\div u\|^2\right),\qquad
\vec{\mathrm{k}}_\mathrm{Lad}\le\frac{3\pi}{32}\,.
\end{equation}
\end{corollary}
\begin{proof}
The proof is based on the equivalence
\eqref{LT}$_{\gamma=1}\Leftrightarrow$\eqref{orth} with equality
for the constants~\eqref{kL} and the fact that the eigenvalues of
the vector Schr\"odinger operator on $\mathbb{S}^2$
\begin{equation}\label{vecSchr}
Av=-\mathbf{\Delta} v-Vv
\end{equation}
have even multiplicities as the following equality implies
(see~\eqref{vecLap}, \eqref{divrot})
$$
\mathbf{\Delta}(v^\perp)=(\mathbf{\Delta} v)^\perp.
$$
Now let $u$ in \eqref{vecLad} be normalized, $\|u\|=1$,  let
$V(s)=\alpha|u(s)|^2$, $\alpha>0$, and let $E$ be the lowest
eigenvalue of~\eqref{vecSchr}.  If $E<0$, then since $E$ is counted
at least twice in the sum $\sum_{\lambda_j\le0}\lambda_j$, it
follows that
\begin{equation}\label{E}
E\ge \frac12\sum_{\lambda_j\le0}\lambda_j\ge-\frac12\mathrm{L_1}\int_{\mathbb{S}^2}V(s)^2dS=
-\alpha^2\frac12\mathrm{L_1}\|u\|_{L_4}^4,
\end{equation}
where the second inequality is~\eqref{LT} with
$$
\mathrm{L_1}\le \frac14\cdot\frac{3\pi}{16},
$$
in view of \eqref{kL} and \eqref{orthvec}. If $E\ge0$, then
\eqref{E} also formally holds.

Next, by the variational principle
\begin{equation}\label{Eless}
\aligned
E\le(Au,u)=\|\rot u\|^2+\|\div u\|^2-\int_{\mathbb{S}^2}V(s)|u(s)|^2dS=\\
\|\rot u\|^2+\|\div u\|^2-\alpha\|u\|^4_{L_4}.
\endaligned
\end{equation}
Combining \eqref{E} and \eqref{Eless} and setting optimal
$\alpha=1/\mathrm{L}_1$ we finally obtain
$$
\|u\|_{L_4}^4\le 2\mathrm{L}_1\left(\|\rot u\|^2+\|\div u\|^2\right)\le
\frac{3\pi}{32}\left(\|\rot u\|^2+\|\div u\|^2\right).
$$
\end{proof}

\begin{remark}
{\rm
It is worth pointing out that
 and for any domain on the sphere
$\Omega\subseteq\mathbb{S} ^2$ and an orthonormal family
$\{u_j\}_{j=1}^N\in H^1_0(\Omega,T\mathbb{S}^2)$ extension by zero
shows that the corresponding Lieb--Thirring constants are uniformly
bounded by the constants on the whole sphere whose estimates were
found  in Theorem~\ref{Th:LT-vec} and Corollary~\ref{C:vec}. }
\end{remark}

\setcounter{equation}{0}
\section{Lieb--Thirring inequalities on  $\mathbb{T}^2$ }\label{sec3}
We now prove Theorem~\ref{Th:1} for the 2D torus. We first consider the torus  with equal periods
and without loss of generality
we set $\mathbb{T}^2=[0,2\pi]^2$.
\begin{proof}[Proof of Theorem~\ref{Th:1} for $\mathbb{T}^2$]
We use the Fourier series
$$
\psi(x)=\frac1{2\pi}\sum_{k\in\mathbb{Z}^2_0}\psi_k e^{ik\cdot x},\qquad
\psi_k=\frac1{2\pi}\int_{\mathbb{T}^2}\psi(x)e^{-ik\cdot x}dx,\quad
\mathbb{Z}^2_0=\mathbb{Z}^2\setminus\{0,0\},
$$
so that
$$
\|\psi\|^2=\sum_{k\in\mathbb{Z}^2_0}|\psi_k|^2,
\qquad
\|\nabla\psi\|^2=\sum_{k\in\mathbb{Z}^2_0}|k|^2|\psi_k|^2.
$$
Then as before we have
$$
\|\nabla\psi\|^2=
\int_0^\infty \sum_{k\in\mathbb{Z}^2_0} f\biggl(\frac E{|k|^2}\biggr)^2|\psi_k|^2dE=
\int_{\mathbb{T}^2}\int_0^\infty|\psi^E(x)|^2dEdx,
$$
where
$$
\psi^E(x)=\frac1{2\pi}\sum_{k\in\mathbb{Z}^2_0} f\biggl(\frac E{|k|^2}\biggr)\psi_ke^{ik\cdot x},
$$
and therefore
$$
\psi(x)-\psi^E(x)=
\frac1{2\pi}\sum_{k\in\mathbb{Z}^2_0}
\left(1-f\biggl(\frac E{|k|^2}\biggr)\right)\psi_ke^{ik\cdot x}=
(\psi(\cdot), \chi^E(\cdot,x)),
$$
where
$$
\chi^E(x',x)=\frac1{2\pi}
\sum_{k\in\mathbb{Z}^2_0}
\left(1-f\biggl(\frac E{|k|^2}\biggr)\right)e^{ik\cdot x'}e^{-ik\cdot x}.
$$
With the  choice of $f$ given in \eqref{choice} and setting $a=\sqrt{\mu}E$ below, we have
\begin{equation}\label{tor}
\aligned
\|\chi^E(\cdot,x)\|^2=\frac1{4\pi^2}\sum_{k\in\mathbb{Z}^2_0}
\left(1-f\biggl(\frac E{|k|^2}\biggr)\right)^2=\\=
\frac1{4\pi^2}\sum_{k\in\mathbb{Z}^2_0}\frac1{\left(\left(\frac{|k|}{\sqrt{a}}\right)^4+1\right)^2}
<\frac a{16}=\frac\pi{64}E=:AE,
\endaligned
\end{equation}
where the key inequality for series is proved in the Appendix.

At this point we can complete the proof as in Theorem~\ref{Th:S2}.
\end{proof}

\subsection{Elongated torus.}
We now briefly discuss the Lieb--Thirring constant on a 2D torus
with aspect ratio $\alpha$. Since the Lieb--Thirring constant
depends only on $\alpha$, we consider the torus $\mathbb{T}^2_\alpha=[0,2\pi/\alpha]\times
[0,2\pi]$. Furthermore, it suffices to consider the case $\alpha\le1$, since otherwise
we merely interchange the periods.

\begin{theorem}\label{Th:alpha}
The Lieb--Thirring constant on the elongated torus $\mathbb{T}^2_\alpha$
satisfies the bound
\begin{equation}\label{alpha}
\mathrm{k}_\mathrm{LT}(\mathbb{T}^2_\alpha)\le\frac1\alpha\frac{3\pi}{32}\
\text{as}\ \alpha\to0.
\end{equation}
\end{theorem}
\begin{proof} We shall prove \eqref{alpha} under an additional technical assumption that
$k=1/\alpha\in\mathbb{N}$. Given the orthonormal family
$\{\psi_j\}_{j=1}^N \in\dot H^1(\mathbb{T}^2_\alpha)$, we extend each $\psi_j$
 by periodicity in the $x_2$ direction $k$ times, multiply the result by $\sqrt{\alpha}$
 and denote the resulting function defined on the square
 torus $\mathbb{T}^2=[0,2\pi k]^2$ by $\widetilde\psi_j$. Then the family
 $\{\widetilde\psi_j\}_{j=1}^N$ is orthonormal in $L_2(\mathbb{T}^2)$ and
 for $\rho_{\widetilde\psi}(x)=\sum_{j=1}^N|\widetilde\psi_j(x)|^2$
 and $\rho_{\psi}(x)=\sum_{j=1}^N|\psi_j(x)|^2$  it holds
 $$
 \int_{\mathbb{T}^2}\rho_{\widetilde\psi}(x)^2dx=
 \alpha\int_{\mathbb{T}^2_\alpha}\rho_{\psi}(x)^2dx,\qquad
 \int_{\mathbb{T}^2}|\nabla\widetilde\psi_j(x)|^2dx=
 \int_{\mathbb{T}^2_\alpha}|\nabla\psi_j(x)|^2dx,
 $$
 which gives \eqref{alpha}.
 \end{proof}
 \begin{remark}
{\rm
The rate of growth $1/\alpha$ of the Lieb--Thirring constant is sharp
as $\alpha\to0$. To see this we set $N=1$ and consider a function
on $\mathbb{T}^2_\alpha$ depending on the long coordinate $x_1$ only.
For example, let $\psi(x_1,x_2)=\sin(2\pi\alpha x_1)$.
Then $\|\psi\|^4_{L_4}\sim 1/\alpha$ $(=\frac{3\pi^2}{2\alpha})$,
$\|\psi\|^2_{L_2}\sim 1/\alpha$ $(=\frac{2\pi^2}{\alpha})$,
$\|\nabla\psi\|^2_{L_2}\sim \alpha$ $(=2\pi^2\alpha)$.
Therefore $\mathrm{k}_\mathrm{LT}(\mathbb{T}^2_\alpha)\succeq 1/\alpha$
$(\ge\frac1\alpha\frac3{8\pi^2})$.
}
\end{remark}
\begin{remark}
{\rm
The orthogonal complement to the  subspace of functions depending only on the long coordinate
$x_1$ consists of functions $\psi(x_1,x_2)$ with mean value
zero with respect to the short coordinate $x_2$:
\begin{equation}\label{T2cond}
\int_0^{2\pi}\psi(x_1,x_2)dx_2=0\quad\forall x_1\in[0,2\pi/\alpha].
\end{equation}
The Lieb--Thirring constant on this subspace is bounded uniformly
with respect to $\alpha$ as $\alpha\to0$. The similar
result holds for the multidimensional torus with different periods.
See \cite{I-L-MS} for the details.
}
\end{remark}

\begin{remark}
{\rm The lifting argument of \cite{Lap-Weid} was used in
\cite{I-L-MS} to derive the Lieb--Thirring inequalities on the
multidimensional with pointwise orthogonality condition of the type
\eqref{T2cond}.   It is not clear how to use the lifting argument
in the case of a global (and weaker) orthogonality contition
$\int_{\mathbb{T}^d}\psi(x)dx=0$.

Finally, we do not know whether the lifting argument
can in some form be used for the Lieb--Thirring inequalities
on the sphere, say, when going over from $\mathbb{S}^{d-1}$ to
$\mathbb{S}^{d}$. }
\end{remark}

\setcounter{equation}{0}
\section{Appendix. Estimates of the series}\label{sec4}

\subsection*{Estimate for the sphere.} The series estimated in \eqref{series1} is precisely of the type
\begin{equation}\label{G}
G(\nu):=\sum_{n=1}^\infty(2n+1)g\left(\nu\, n(n+1)\right),
\end{equation}
where $g$ is sufficiently smooth and sufficiently fast decays at
infinity. We need to find the asymptotic behavior of $G(\nu)$ as
$\nu\to0$. This has been done in~\cite{IZ} where the following
result was proved.
\begin{lemma}\label{L:E-M}
The following asymptotic expansion holds as $\nu\to0$:
\begin{equation}\label{as2}
G(\nu)=\frac1{\nu}\int_0^\infty g(t)dt-\frac23g(0)-
\frac1{15}\nu g'(0)+\frac4{315}\nu^2g''(0)+
O(\nu^3).
\end{equation}
\end{lemma}

The series in \eqref{series1} is of the form~\eqref{G}
with
$$
g(t)=\frac1{(1+t^2)^2}, \qquad \nu=\frac1{\sqrt{\mu}E},
$$
so that $g(0)=1$, $g'(0)=0$, $g''(0)=-4$ and $\int_0^\infty g(t)dt=\pi/4$. Therefore
\eqref{as2} gives
$$
\sum_{n=1}^\infty\frac
{(2n+1)}{\biggl({1+\left(\frac{n(n+1)}a\right)^2}\biggr)^2}=
    a\frac\pi 4-\frac23-\frac{16}{315}a^{-2}+O(a^{-3}),\ \text{as}\ a\to\infty,
$$
which shows that
inequality  \eqref{series1}
clearly holds for large energies $E>E_0$. The proof of inequality \eqref{series1}  for
all $E\in[0,\infty)$ amounts to showing that the inequality
$$
H_{\mathbb{S}^2}(a):=\frac4\pi\,a^3\sum_{n=1}^\infty\frac{2n+1}{\bigl(\bigl(n(n+1)\bigr)^2+a^2\bigr)^2}
\,<\,1, \quad a=\sqrt{\mu}E=\frac\pi 4E
$$
holds on a \emph{finite} interval $ a\in[0,a_0]$.
The value of $a_0$ (say, $20$) can be specified similarly to \cite{I12JST}.
Furthermore,  the sum of the series
can be found in an explicit form in terms of the (digamma) $\psi$-function.
The function $H_{\mathbb{S}^2}(a)$ and the third-order remainder term are shown in Fig.~\ref{fig:S2}.
\begin{figure}[htb]
\centerline{\psfig{file=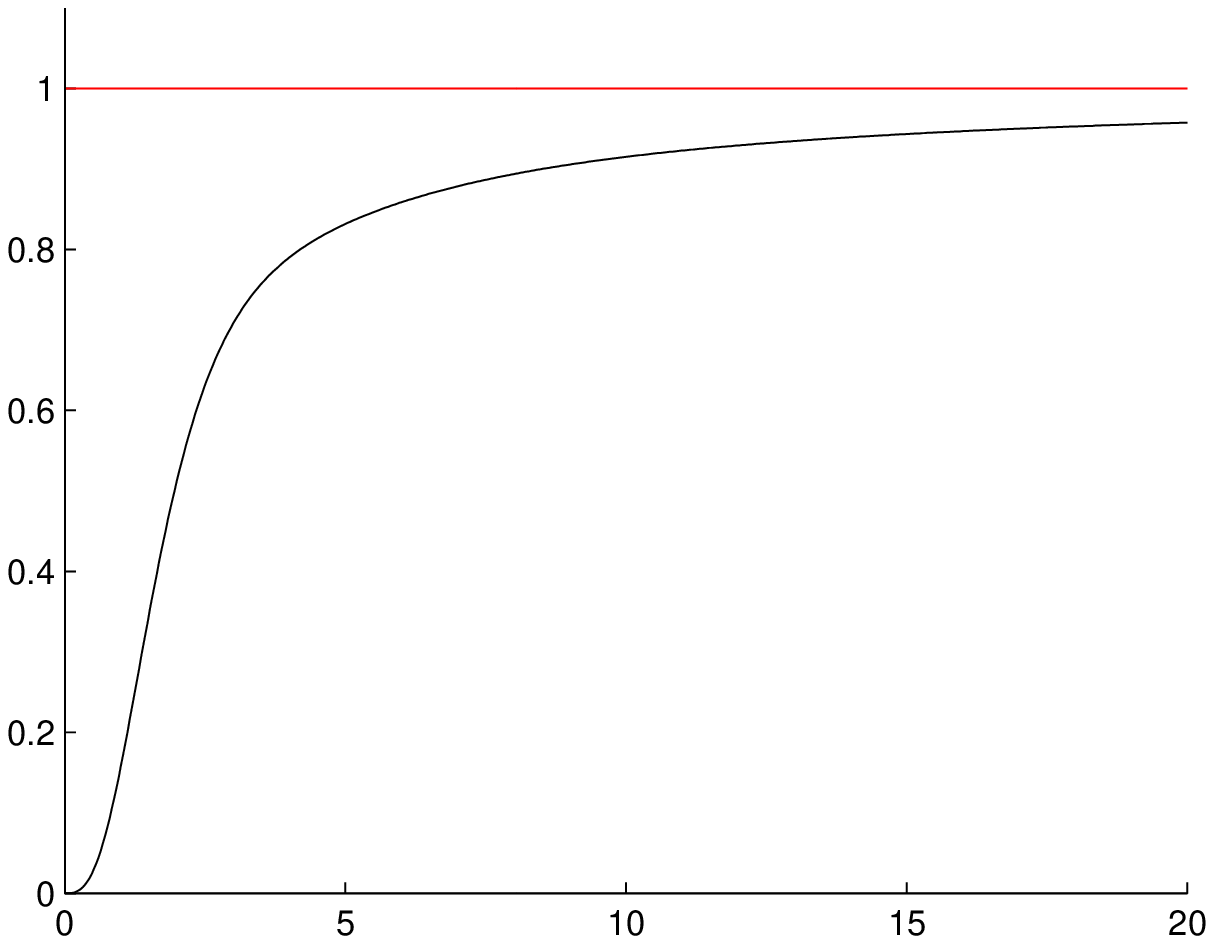,width=7.5cm,height=6cm,angle=0}
\psfig{file=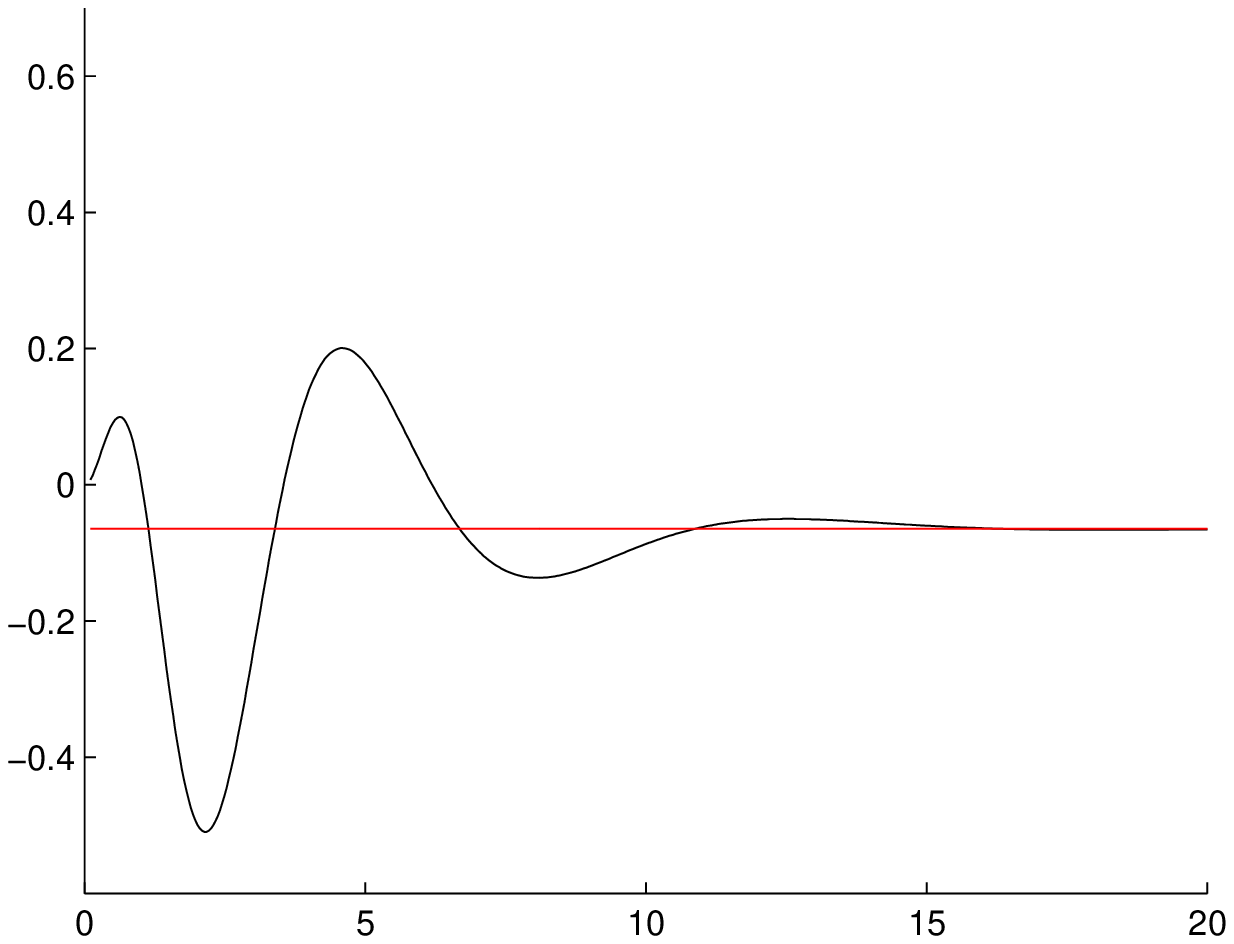,width=7.5cm,height=6cm,angle=0}}
\caption{The graph of $H_{\mathbb{S}^2}(a)$ is on the left;
the remainder term $\left(H_{\mathbb{S}^2}(a)-1-\frac8{3\pi a}\right)\cdot a^3$
is shown on the right, the horizontal red line is  $-64/(315\pi)=-0.064$.}
\label{fig:S2}
\end{figure}
\subsection*{Estimate for the torus.}

\begin{lemma}\label{L:Poisson}
The following asymptotic expansion holds as $a\to\infty$:
\begin{equation}\label{at2}
\sum_{k\in\mathbb{Z}^2_0}\frac1{\left(\left(\frac{|k|}{\sqrt{a}}\right)^4+1\right)^2}
=\frac{\pi^2}4 a-1+O(e^{-C\sqrt{a}}).
\end{equation}
\end{lemma}
\begin{proof}
We use the Poisson summation formula
(see, e.\,g., \cite{S-W}):
$$
\sum_{m\in\mathbb{Z}^d}f(m/\nu)=
(2\pi)^{d/2}\nu^d
\sum_{m\in\mathbb{Z}^d}\widehat{f}(2\pi m \nu),
$$
where
$\widehat{f}(\xi)=(2\pi)^{-d/2}\int_{\mathbb{R}^d}
f(x)e^{-i\xi x}dx$.

Taking into account that the term with $k=(0,0)$ is missing in
\eqref{at2}, the Poisson summation formula gives
\begin{equation}\label{expsmall}
\aligned
\sum_{k\in\mathbb{Z}^2_0}\frac1{\left(\left(\frac{|k|}{\sqrt{a}}\right)^4+1\right)^2}+1=\\=
a\int_{\mathbb{R}^2}\frac{dxdy}{((x^2+y^2)^2+1)^2}+2\pi a\sum_{k\in\mathbb{Z}^2_0}
\widehat h(2\pi\sqrt{a}|k|)=\frac{\pi^2}4a+O(e^{-C\sqrt{a}}),
\endaligned
\end{equation}
where $h(x,y)=\frac1{((x^2+y^2)^2+1)^2}$ is analytic
and therefore its Fourier transform
\begin{equation}\label{Fourier}
\widehat h(\xi)=\frac1{2\pi}\int_{\mathbb{R}^2}
e^{-ix\xi_1-iy\xi_2}h(x,y)dxdy
\end{equation}
is exponentially decaying.
\end{proof}

The graph of the function
\begin{equation}\label{lessthen1}
H_{\mathbb{T}^2}(a):=\frac4{\pi^2}\,a^3\sum_{k\in\mathbb{Z}^2_0}\frac{1}{\bigl(|k|^4+a^2\bigr)^2}
\,<\,1, 
\end{equation}
and the exponentially small remainder term
$$
R(a)=2\pi a\sum_{k\in\mathbb{Z}^2_0}
\widehat h(2\pi\sqrt{a}|k|)
$$
are shown in Fig.~\ref{fig:T2}
\begin{figure}[htb]
\centerline{\psfig{file=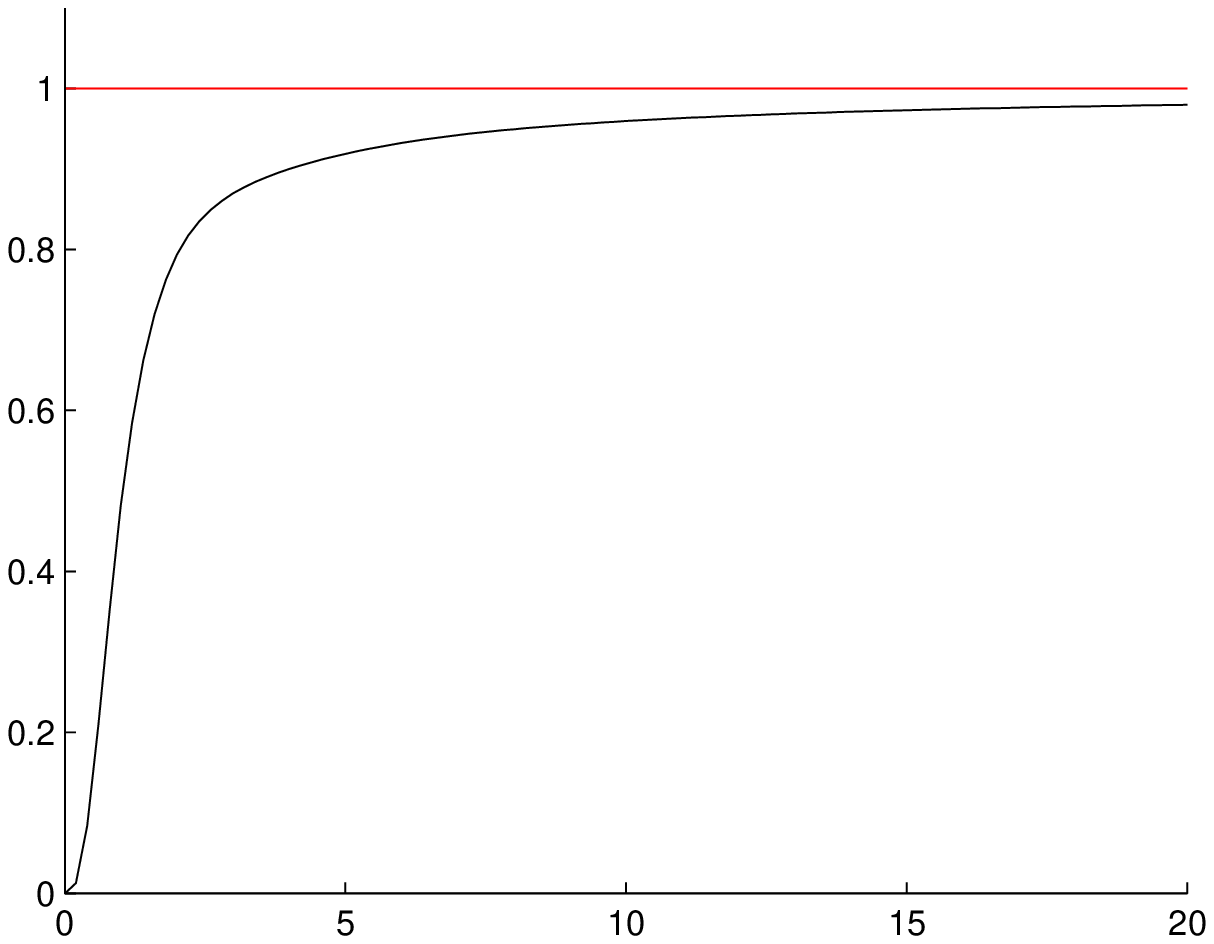,width=7.5cm,height=6cm,angle=0}
\psfig{file=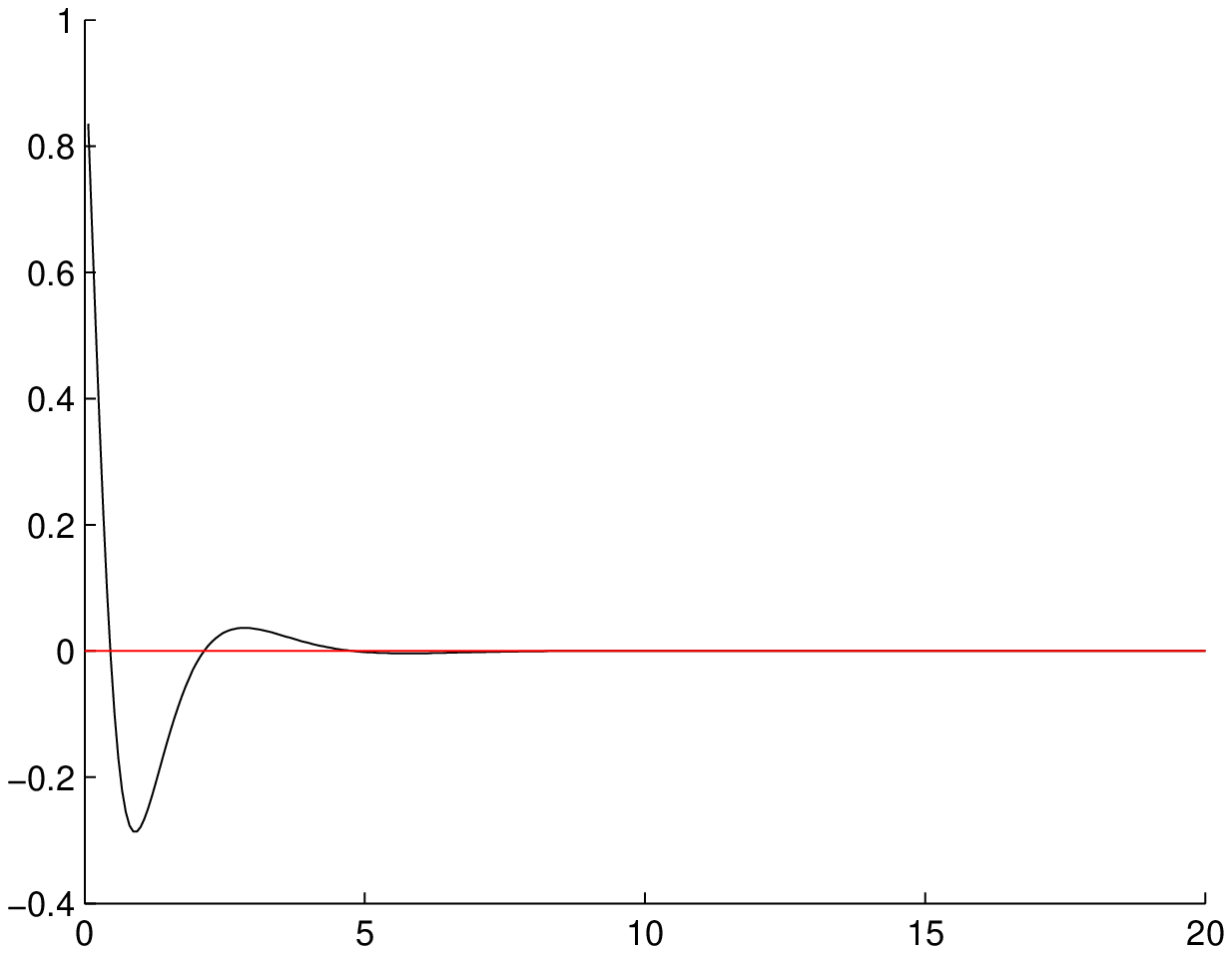,width=7.5cm,height=6cm,angle=0}}
\caption{The function $H_{\mathbb{T}^2}(a)$ is shown on the left;
the exponentially small term $R(a)$ is shown on the right.}
\label{fig:T2}
\end{figure}

We now give an explicit estimate for the exponentially small remainder term in~\eqref{expsmall}.
The function $h(z)$ is analytic in the domain $\Omega\subset\mathbb{C}^2$:
$\Omega=\{z\in \mathbb{C}^2,\ |\mathrm{Im} z_1|<\frac12,\ |\mathrm{Im} z_2|<\frac12\}$.
In fact, the equation
$$
(x+i\alpha)^2+(y+i\alpha)^2=\pm i
$$
has real solutions $x$ and $y$ only for $\alpha\ge\frac12$.

For $F(x,y,\alpha)=((x+i\alpha)^2+(y+i\alpha)^2)^2+1$ we have
$$
\aligned
&\mathrm{Re}\,F=
(x^2+y^2)^2-8\alpha^2(x^2+y^2+xy)+4\alpha^4+1\ge t^2-12\alpha^2t+4\alpha^4+1,\\
&\mathrm{Im}\,F=
4a(x+y)(x^2+y^2-2\alpha^2),\quad|\mathrm{Im}\,F|\le4\sqrt{2}\alpha t^{1/2}|t-2\alpha^2|,
\endaligned
$$
where $t:=x^2+y^2$. Next, by a direct inspection we verify that
for $t\ge0$
$$
\aligned
|F^2|\ge
 (\mathrm{Re}\,F)^2-(\mathrm{Im}\,F)^2=\\
(t^2-12\alpha^2t+4\alpha^4+1)^2-32\alpha^2t(t-2a^2)^2>\frac1b(t^4+1),
\endaligned
$$
where $\alpha=4.6^{-1}$ and $b=4.75$.
This gives that for $|\mathrm{Im} z_1|\le\alpha,\ |\mathrm{Im} z_2|\le\alpha$
$$
|h(x+i\alpha,y+i\alpha)|\le\frac b{(x^2+y^2)^4+1}\,.
$$
By the Cauchy integral theorem we can shift the $x$ and $y$ integration in \eqref{Fourier}
in the complex plane by $\pm i\alpha$ (depending on the sign of $\xi_1$ and $\xi_2$)  and find that
$$
|\widehat h(\xi)|\!\le\!\frac b{2\pi}e^{-(|\xi_1|+|\xi_2|)\alpha}\!\int_{\mathbb{R}^2}
\frac{dxdy}{(x^2+y^2)^4+1}\!=e^{-(|\xi_1|+|\xi_2|)\alpha}\frac{b\pi\sqrt{2}}8\le
e^{-\alpha|\xi|}\frac{b\pi\sqrt{2}}8\,.
$$

We write  the numbers $|k|^2$ over the lattice $\mathbb{Z}^2_0$
in non-decreasing order and denote them by $\{\lambda_j\}^\infty_{j=1}$.
Using that $\lambda_j\ge j/4$ (see \cite{I-L-AA}) and setting
$L:=\frac{\alpha\pi\sqrt{a}}2$ and $A:=\frac{\pi^2\sqrt{2}ab}4$ we
estimate the series in \eqref{expsmall} as follows
$$
\aligned
|R(a)|\le
2\pi a\sum_{k\in\mathbb{Z}^2_0}
|\widehat h(2\pi\sqrt{a}|k|)|=2\pi a\sum_{j=1}^\infty
|\widehat h(2\pi\sqrt{a}\lambda_j^{1/2})|\le\\
A\sum_{j=1}^\infty
e^{-2\pi\alpha\sqrt{a}\lambda_j^{1/2}}\le
A\sum_{j=1}^\infty
e^{-2Lj^{1/2}}=Ae^{-L}\sum_{j=1}^\infty e^{-L(2j^{1/2}-1)}\le\\
Ae^{-L}\sum_{j=1}^\infty e^{-Lj^{1/2}}<
Ae^{-L}\int_0^\infty e^{-L\sqrt{x}}dx=
Ae^{-L}\frac2{L^2}=\frac{2^{3/2}b}{\alpha^2}e^{-\frac{\alpha\pi\sqrt{a}}2}\,.
\endaligned
$$
Inequality~\eqref{lessthen1} holds if $R(a)<1$. The above estimates show that
$|R(a)|<1$ for
$$
a>\left[\frac2{\alpha\pi}\log\left(\frac{2^{3/2}b}{\alpha^2}\right)\right]^2=
273.8\,.
$$

A more optimistic estimate follows from the fact that $h(x)$ is radial and therefore
so is its Fourier transform
 $$\widehat h(\xi)=\int_0^\infty J_0(|\xi|r)h(r)rdr,
$$
where $J_0$ is the Bessel function. The latter integral is expressed in terms of the
Meijer G-function  and satisfies $|\widehat h(\xi)|<e^{-|\xi|/2}$. Similar estimates
give that $|R(a)|<1$ for
$$
a>\left[\frac4\pi\log\frac{64}\pi\right]^2=14.73\,.
$$

\subsection*{Acknowledgements}\label{SS:Acknow}
The work of A.\,I. and S.\,Z. is supported in part by the  Russian
Science Foundation grant 19-71-30004 (sections 1,2). Research of A.\,L. is
supported by the Russian Science Foundation grant 19-71-30002 (sections 3,4).

\end{document}